\documentclass[11pt]{article}
\usepackage{razb}
\usepackage{color}

\newcommand{\refeq}[1]{{\rm (\ref{#1})}}
\newlength{\longeqmarginwidth}
\newlength{\longeqwidth}
\newlength{\longeqskiplength}
\newcommand{\Vzero}[2]{\mathord{\stackrel{\textstyle\kern1pt\circ}
{\smash V\vbox to6pt{}}}\vphantom{V}_{#1}^{#2}}

\newcommand{\function}[2]{:#1 \longrightarrow #2}
\newcommand{\of}[1]{\left( #1 \right)}

\newcommand{\set}[2]{\left\{\hspace{0.2ex} #1 \left|\: #2
\right. \right\}}

\newcommand{\eval}[2]{\llbracket #1 \rrbracket_{#2}}

\newcommand{\df}{\stackrel{\rm def}{=}}

\newcommand{\im}{{\rm im}}

\newcommand{\Hom}{{\rm Hom}}

\newcommand{\scr}{\mathcal}

\newcounter{operator}

\usepackage{amsfonts}
\usepackage{amsmath}
\usepackage{epic,eepic}
\usepackage{stmaryrd}
\def\veca{\mathbf}
\def\vecb{\mathbf}
\title{On the Number of Pentagons in Triangle-Free Graphs}
\author{
Hamed Hatami\thanks{School of Computer Science, McGill University, {\tt hatami@cs.mcgill.ca}.} \and
Jan Hladk\'y\thanks{Mathematics Institute \emph{and} DIMAP, University of Warwick, {\tt honzahladky@gmail.com}. Supported by EPSRC award EP/D063191/1.} \and
Daniel Kr\'al'\thanks{Institute of Mathematics, DIMAP and Department of Computer Science, University of Warwick, Coventry CV4 7AL, United Kingdom. Previous affiliation: Institute of Computer Science (IUUK), Faculty of Mathematics and Physics, Charles University, Malostransk\'e n\'am\v est\'\i{} 25, 118 00 Prague 1, Czech Republic. E-mail: \texttt{D.Kral@warwick.ac.uk}. The work of this author leading to this invention has received funding from the European Research Council under the European Union's Seventh Framework Programme (FP7/2007-2013)/ERC grant agreement no.~259385.} \and
Serguei Norine\thanks{Department of Mathematics \& Statistics, McGill
University, {\tt snorin@math.mcgill.ca}.} \and
Alexander Razborov\thanks{University of Chicago, {\tt razborov@cs.uchicago.edu}. Part of this work was done while the author was at Steklov Mathematical Institute, supported by the Russian Foundation
for Basic Research, and at Toyota Technological Institute, Chicago.}
}
\date{}
\begin{document}
\maketitle
\begin{abstract}
Using the formalism of flag algebras,
we prove that
every triangle-free graph $G$ with $n$ vertices contains at most $(n/5)^5$ cycles of length five.
Moreover, the equality is attained only when $n$ is divisible by five and $G$ is the balanced
blow-up of the pentagon. We also compute the maximal number of pentagons and characterize extremal graphs in the non-divisible case provided $n$ is sufficiently large.
This settles a conjecture made by Erd\H{o}s in 1984.
\end{abstract}

\section{Introduction}

Triangle-free graphs need not be bipartite. But how exactly far from being bipartite can they be?
In 1984, Erd\H{o}s \cite[Questions 1 and 2]{Erd3} considered three quantitative refinements of this question. More precisely, he proposed to measure ``non-bipartiteness'' by:
\begin{itemize}
\item[{\bf (i)}]  the minimal possible number of edges in a subgraph spanned by half of the vertices;

\item[{\bf (ii)}] the minimal possible number of edges that have to be removed to make the graph bipartite;

\item[{\bf (iii)}] the number of copies of pentagons (cycles of length $5$) in the graph.
\end{itemize}

All these parameters vanish on bipartite graphs, and Erd\H{o}s conjectured that in the class of triangle-free graphs every one of them is maximized by balanced blow-ups of the pentagon. Simonovits (referred to in \cite{Erd3}) observed that another example which attains the conjectured extremum for~{\bf (i)} is provided by balanced blow-ups of the Petersen graph.
For {\bf (iii)}, Michael \cite{Mic} noticed that the cycle of length eight with four chords joining the opposite vertices
has eight pentagons, thus matching the number of pentagons in the eight-vertex (almost balanced) blow-ups
of the pentagon.

The first two of Erd\H{o}s's questions have been investigated in~\cite{EFPSBipartite,KriTriangleFree,KeeSudSparse}.
Gy\H{o}ri investigated the third question in~\cite{Gyo}. In
terms of densities, Erd\H{o}s's conjecture regarding {\bf (iii)} states that the
density of pentagons in any triangle-free graphs is at most $\frac{5!}{5^5}$.
Gy\H{o}ri proved an upper bound of $\frac{3^3\cdot 5!}{5\cdot
2^{14}}$ that is within a factor 1.03 of the optimal. F\"uredi (personal
communication) refined Gy\H{o}ri's approach and obtained an upper bound within a
factor 1.001 of the optimal.

In this paper we settle {\bf (iii)} in the density sense  (Theorem~\ref{main}), which also implies the exact solution when $n$ is divisible by 5 (Corollary~\ref{cor-main}).
The proof of this result
is a calculation in flag algebras (introduced in~\cite{flag}).
Furthermore, we obtain the asymptotic uniqueness (that, again, turns into the uniqueness in the ordinary sense when $5|n$) by a relatively
simple argument in Theorem~\ref{asymptotic_uniqueness}. In Section~\ref{sec:Exact} we use a more technical approach to find the exact solution for $n$ sufficiently large.
We leave it as an open question to prove an exact bound on the maximum number of pentagons
in an $n$-vertex triangle-free graph for all values of  $n$.

We assume a certain familiarity with the theory of flag algebras
from~\cite{flag}. For the proof of the central Theorem~\ref{main} only the most basic notions which deal with Cauchy-Schwarz type calculations are required. Thus, instead of trying to duplicate definitions, we occasionally give pointers to relevant places in~\cite{flag} and some subsequent papers. For our proof of Theorem~\ref{asymptotic_uniqueness} we need a little bit more than these basics. We recall the corresponding bits in Section~\ref{sec:MoreAdvancedFlags}.

\section{Preliminaries} \label{prel}

\subsection{Notation} \label{sec:notation}
We denote vectors with bold font, e.g. $\veca a=(\vecb a(1),\veca a(2), \veca a(3))$ is a vector with three coordinates.
For every positive integer $k$, let $[k]$ denote the set $\{1,\ldots,k\}$.  Following \cite[Definition 1]{flag}, for two graphs $H$ and $G$,  the density of $H$ in $G$  as an induced subgraph is denoted by $p(H,G)$. That is $p(H,G)$ is the probability that the subgraph induced on $|V(H)|$ randomly chosen vertices of $G$ is isomorphic to $H$. 

Except for the stand-alone Section~\ref{blowup}, we exclusively work \cite[\S
2]{flag} in the theory $T_{\mathrm{TF-Graph}}$ of triangle-free graphs. Recall
from~\cite{flag} that for a theory $T$ and a positive integer $n$, the set of
all models of $T$ on $n$ elements up to an isomorphism is denoted by
$\mathcal{M}_n[T]$. We work with the notion of types, flags, and flag algebras, and use the same notation as in \cite[\S 2.1]{flag} where this terminology  is introduced. Let us list those models, types and flags  that will be needed in this paper. 

Let $\rho\in \mathcal{M}_2[T_{\mathrm{TF-Graph}}]$ and $C_5\in \mathcal{M}_5[T_{\mathrm{TF-Graph}}]$ respectively denote the edge and the pentagon. These two graphs along with two other graphs that will be needed for proving the uniqueness and the exact result are illustrated in Figure \ref{models}.
\begin{figure}[tbp]
\begin{center}
\input{models.eepic}
\caption{\label{models} Models}
\end{center}
\end{figure}

We denote the {\em trivial type} of size 0 by $0$. Let $P$ denote the type of size 5 based on $C_5$ (see Figure \ref{types}). For $i=0,1,2$, let $\sigma_i$ denote the type of size 3 with $i$ edges where the labeling is chosen in such a way that the permutation of 1 and 2 is an automorphism (see Figure \ref{types}).
\begin{figure}[tbp]
\begin{center}
\setlength{\unitlength}{0.254mm}
\begin{picture}(424,92)(13,-86)
        \special{color rgb 0 0 0}\allinethickness{0.254mm}\special{sh 0.99}\put(15,-65){\ellipse{4}{4}} 
        \special{color rgb 0 0 0}\allinethickness{0.254mm}\special{sh 0.99}\put(40,-15){\ellipse{4}{4}} 
        \special{color rgb 0 0 0}\allinethickness{0.254mm}\special{sh 0.99}\put(65,-65){\ellipse{4}{4}} 
        \special{color rgb 0 0 0}\allinethickness{0.254mm}\special{sh 0.99}\put(115,-65){\ellipse{4}{4}} 
        \special{color rgb 0 0 0}\allinethickness{0.254mm}\special{sh 0.99}\put(140,-15){\ellipse{4}{4}} 
        \special{color rgb 0 0 0}\allinethickness{0.254mm}\special{sh 0.99}\put(165,-65){\ellipse{4}{4}} 
        \special{color rgb 0 0 0}\put(35,-86){\shortstack{$\sigma_0$}} 
        \special{color rgb 0 0 0}\put(135,-86){\shortstack{$\sigma_1$}} 
        \special{color rgb 0 0 0}\allinethickness{0.254mm}\special{sh 0.99}\put(215,-65){\ellipse{4}{4}} 
        \special{color rgb 0 0 0}\allinethickness{0.254mm}\special{sh 0.99}\put(240,-15){\ellipse{4}{4}} 
        \special{color rgb 0 0 0}\allinethickness{0.254mm}\special{sh 0.99}\put(265,-65){\ellipse{4}{4}} 
        \special{color rgb 0 0 0}\put(235,-86){\shortstack{$\sigma_2$}} 
        \special{color rgb 0 0 0}\allinethickness{0.254mm}\path(215,-65)(240,-15) 
        \special{color rgb 0 0 0}\allinethickness{0.254mm}\path(240,-15)(265,-65) 
        \special{color rgb 0 0 0}\allinethickness{0.254mm}\path(115,-65)(165,-65) 
        \special{color rgb 0 0 0}\put(20,-61){\shortstack{\scriptsize 1}} 
        \special{color rgb 0 0 0}\put(120,-61){\shortstack{\scriptsize 1}} 
        \special{color rgb 0 0 0}\put(220,-66){\shortstack{\scriptsize 1}} 
        \special{color rgb 0 0 0}\put(55,-61){\shortstack{\scriptsize 2}} 
        \special{color rgb 0 0 0}\put(155,-61){\shortstack{\scriptsize 2}} 
        \special{color rgb 0 0 0}\put(255,-66){\shortstack{\scriptsize 2}} 
        \special{color rgb 0 0 0}\put(45,-16){\shortstack{\scriptsize 3}} 
        \special{color rgb 0 0 0}\put(145,-16){\shortstack{\scriptsize 3}} 
        \special{color rgb 0 0 0}\put(245,-16){\shortstack{\scriptsize 3}} 
        \special{color rgb 0 0 0}\allinethickness{0.254mm}\special{sh 0.99}\put(325,-65){\ellipse{4}{4}} 
        \special{color rgb 0 0 0}\allinethickness{0.254mm}\special{sh 0.99}\put(365,-65){\ellipse{4}{4}} 
        \special{color rgb 0 0 0}\allinethickness{0.254mm}\special{sh 0.99}\put(375,-30){\ellipse{4}{4}} 
        \special{color rgb 0 0 0}\allinethickness{0.254mm}\special{sh 0.99}\put(315,-30){\ellipse{4}{4}} 
        \special{color rgb 0 0 0}\allinethickness{0.254mm}\special{sh 0.99}\put(345,-5){\ellipse{4}{4}} 
        \special{color rgb 0 0 0}\allinethickness{0.254mm}\path(315,-30)(345,-5) 
        \special{color rgb 0 0 0}\allinethickness{0.254mm}\path(345,-5)(375,-30) 
        \special{color rgb 0 0 0}\allinethickness{0.254mm}\special{sh 0.99}\put(325,-65){\ellipse{4}{4}} 
        \special{color rgb 0 0 0}\allinethickness{0.254mm}\special{sh 0.99}\put(365,-65){\ellipse{4}{4}} 
        \special{color rgb 0 0 0}\allinethickness{0.254mm}\special{sh 0.99}\put(375,-30){\ellipse{4}{4}} 
        \special{color rgb 0 0 0}\allinethickness{0.254mm}\special{sh 0.99}\put(315,-30){\ellipse{4}{4}} 
        \special{color rgb 0 0 0}\allinethickness{0.254mm}\special{sh 0.99}\put(345,-5){\ellipse{4}{4}} 
        \special{color rgb 0 0 0}\allinethickness{0.254mm}\path(315,-30)(345,-5) 
        \special{color rgb 0 0 0}\allinethickness{0.254mm}\path(345,-5)(375,-30) 
        \special{color rgb 0 0 0}\allinethickness{0.254mm}\special{sh 0.99}\put(325,-65){\ellipse{4}{4}} 
        \special{color rgb 0 0 0}\allinethickness{0.254mm}\special{sh 0.99}\put(365,-65){\ellipse{4}{4}} 
        \special{color rgb 0 0 0}\allinethickness{0.254mm}\special{sh 0.99}\put(375,-30){\ellipse{4}{4}} 
        \special{color rgb 0 0 0}\allinethickness{0.254mm}\special{sh 0.99}\put(315,-30){\ellipse{4}{4}} 
        \special{color rgb 0 0 0}\allinethickness{0.254mm}\special{sh 0.99}\put(345,-5){\ellipse{4}{4}} 
        \special{color rgb 0 0 0}\allinethickness{0.254mm}\path(315,-30)(345,-5) 
        \special{color rgb 0 0 0}\allinethickness{0.254mm}\path(345,-5)(375,-30) 
        \special{color rgb 0 0 0}\allinethickness{0.254mm}\special{sh 0.99}\put(325,-65){\ellipse{4}{4}} 
        \special{color rgb 0 0 0}\allinethickness{0.254mm}\special{sh 0.99}\put(365,-65){\ellipse{4}{4}} 
        \special{color rgb 0 0 0}\allinethickness{0.254mm}\special{sh 0.99}\put(375,-30){\ellipse{4}{4}} 
        \special{color rgb 0 0 0}\allinethickness{0.254mm}\special{sh 0.99}\put(315,-30){\ellipse{4}{4}} 
        \special{color rgb 0 0 0}\allinethickness{0.254mm}\special{sh 0.99}\put(345,-5){\ellipse{4}{4}} 
        \special{color rgb 0 0 0}\allinethickness{0.254mm}\path(315,-30)(345,-5) 
        \special{color rgb 0 0 0}\allinethickness{0.254mm}\path(345,-5)(375,-30) 
        \special{color rgb 0 0 0}\allinethickness{0.254mm}\path(325,-65)(315,-30) 
        \special{color rgb 0 0 0}\allinethickness{0.254mm}\path(365,-65)(375,-30) 
        \special{color rgb 0 0 0}\allinethickness{0.254mm}\special{sh 0.99}\put(325,-65){\ellipse{4}{4}} 
        \special{color rgb 0 0 0}\allinethickness{0.254mm}\special{sh 0.99}\put(365,-65){\ellipse{4}{4}} 
        \special{color rgb 0 0 0}\allinethickness{0.254mm}\special{sh 0.99}\put(375,-30){\ellipse{4}{4}} 
        \special{color rgb 0 0 0}\allinethickness{0.254mm}\special{sh 0.99}\put(315,-30){\ellipse{4}{4}} 
        \special{color rgb 0 0 0}\allinethickness{0.254mm}\special{sh 0.99}\put(345,-5){\ellipse{4}{4}} 
        \special{color rgb 0 0 0}\allinethickness{0.254mm}\path(315,-30)(345,-5) 
        \special{color rgb 0 0 0}\allinethickness{0.254mm}\path(345,-5)(375,-30) 
        \special{color rgb 0 0 0}\allinethickness{0.254mm}\special{sh 0.99}\put(325,-65){\ellipse{4}{4}} 
        \special{color rgb 0 0 0}\allinethickness{0.254mm}\special{sh 0.99}\put(365,-65){\ellipse{4}{4}} 
        \special{color rgb 0 0 0}\allinethickness{0.254mm}\special{sh 0.99}\put(375,-30){\ellipse{4}{4}} 
        \special{color rgb 0 0 0}\allinethickness{0.254mm}\special{sh 0.99}\put(315,-30){\ellipse{4}{4}} 
        \special{color rgb 0 0 0}\allinethickness{0.254mm}\special{sh 0.99}\put(345,-5){\ellipse{4}{4}} 
        \special{color rgb 0 0 0}\allinethickness{0.254mm}\path(315,-30)(345,-5) 
        \special{color rgb 0 0 0}\allinethickness{0.254mm}\path(345,-5)(375,-30) 
        \special{color rgb 0 0 0}\allinethickness{0.254mm}\special{sh 0.99}\put(325,-65){\ellipse{4}{4}} 
        \special{color rgb 0 0 0}\allinethickness{0.254mm}\special{sh 0.99}\put(365,-65){\ellipse{4}{4}} 
        \special{color rgb 0 0 0}\allinethickness{0.254mm}\special{sh 0.99}\put(375,-30){\ellipse{4}{4}} 
        \special{color rgb 0 0 0}\allinethickness{0.254mm}\special{sh 0.99}\put(315,-30){\ellipse{4}{4}} 
        \special{color rgb 0 0 0}\allinethickness{0.254mm}\special{sh 0.99}\put(345,-5){\ellipse{4}{4}} 
        \special{color rgb 0 0 0}\allinethickness{0.254mm}\path(315,-30)(345,-5) 
        \special{color rgb 0 0 0}\allinethickness{0.254mm}\path(345,-5)(375,-30) 
        \special{color rgb 0 0 0}\allinethickness{0.254mm}\special{sh 0.99}\put(325,-65){\ellipse{4}{4}} 
        \special{color rgb 0 0 0}\allinethickness{0.254mm}\special{sh 0.99}\put(365,-65){\ellipse{4}{4}} 
        \special{color rgb 0 0 0}\allinethickness{0.254mm}\special{sh 0.99}\put(375,-30){\ellipse{4}{4}} 
        \special{color rgb 0 0 0}\allinethickness{0.254mm}\special{sh 0.99}\put(315,-30){\ellipse{4}{4}} 
        \special{color rgb 0 0 0}\allinethickness{0.254mm}\special{sh 0.99}\put(345,-5){\ellipse{4}{4}} 
        \special{color rgb 0 0 0}\allinethickness{0.254mm}\path(315,-30)(345,-5) 
        \special{color rgb 0 0 0}\allinethickness{0.254mm}\path(345,-5)(375,-30) 
        \special{color rgb 0 0 0}\put(345,-86){\shortstack{$P$}} 
        \special{color rgb 0 0 0}\allinethickness{0.254mm}\path(325,-65)(315,-30) 
        \special{color rgb 0 0 0}\allinethickness{0.254mm}\path(365,-65)(375,-30) 
        \special{color rgb 0 0 0}\allinethickness{0.254mm}\path(325,-65)(365,-65) 
        \special{color rgb 0 0 0}\put(325,-61){\shortstack{\scriptsize 1}} 
        \special{color rgb 0 0 0}\put(370,-66){\shortstack{\scriptsize 2}} 
        \special{color rgb 0 0 0}\put(380,-31){\shortstack{\scriptsize 3}} 
        \special{color rgb 0 0 0}\put(350,-6){\shortstack{\scriptsize 4}} 
        \special{color rgb 0 0 0}\put(320,-36){\shortstack{\scriptsize 5}} 
        \special{color rgb 0 0 0} 
\end{picture}
\caption{\label{types} Types}
\end{center}
\end{figure}



For a type $\sigma$ of size $k$ and an independent set of vertices $V\subseteq [k]$ in $\sigma$, let $F^\sigma_V$ denote the flag $(G,\theta)\in \scr F^\sigma_{k+1}$ in which the only unlabeled vertex $v$ is connected to the set $\set{\theta(i)}{i\in V}$. Note that since we are working in the theory of triangle-free graphs, we have
$$\scr F^\sigma_{k+1} = \{ F^\sigma_V \ | \ \mbox{$V\subseteq [k]$ is an independent set in $\sigma$} \}.$$

\subsection{Operator $\pi^\sigma$, and extension measures}\label{sec:MoreAdvancedFlags}
As usual, the algebra generated by $\sigma$-flags is denoted $\scr A^\sigma$. We then define the following ``upward operator'' $\pi^\sigma:\scr A^0\rightarrow \scr A^\sigma$
as introduced in~\cite[\S 2.3, \S 2.3.1]{flag}. For a triangle-free graph $G\in\scr F^0_{k}\ (= \scr M_k)$, we let $\pi^\sigma(G)$ be the sum of all those
$F=(H,\theta)\in \mathcal{F}^\sigma_{k+|\sigma|}$ for which the unlabeled vertices form a copy of $G$, i.e., $H-\im(\theta)$ is isomorphic to $G$. We then extend $\pi^\sigma$ linearly to $\scr A^0$. It turns out~\cite[Theorem 2.6]{flag} that $\pi^\sigma:\scr A^0\rightarrow \scr A^\sigma$ is an algebra homomorphism.

The next notion we shall need is that of extension measure, as introduced in~\cite[Definition 8]{flag}. Suppose that $\phi\in\Hom^+(\scr A^0[T_{\mathrm{TF-Graph}}], \mathbb{R})$ is such that
\begin{equation}\label{eq:ViewAsunlabelled}
\phi(\sigma)>0
\end{equation}
for a type $\sigma$ (in~\eqref{eq:ViewAsunlabelled} we view $\sigma$ as an unlabeled graph). Then, there exists a unique probability measure $\bf P^\sigma$ on Borel sets of $\Hom^+(\scr A^\sigma[T_{\mathrm{TF-Graph}}], \mathbb{R})$ with the property that
\begin{equation}\label{eq:defExtMeas}
\int_{\Hom^+(\scr A^\sigma[T_{\mathrm{TF-Graph}}], \mathbb{R})}\psi(f)\; {\bf P}^\sigma(d\psi)=\frac{\phi(\eval{f}\sigma)}{\phi(\eval{1_\sigma}\sigma)}\;, \end{equation}
for any $f\in \scr A^\sigma$. See~\cite[Theorem 3.5]{flag}. We say that $\bf P^\sigma$ \emph{extends} $\phi$. $S^\sigma(\phi)$ is the {\em support} of this measure, i.e. the minimal
closed subset $A$ such that ${\bf P}^\sigma[A]=1$ \cite[Section 2.1.1]{fdf}. Note that the integration in \eqref{eq:defExtMeas} can be restricted to $S^\sigma(\phi)$.

Observe that $\phi$ can be reconstructed from its extension ${\bf P}^\sigma$ simply by picking an arbitrary $\psi\in S^\sigma(\phi)$ and letting
\begin{equation}\label{eq:reconstruct}
\phi(g)=\psi(\pi^\sigma(g))\ (g\in \scr A^0)
\end{equation}
(cf.~\cite[Corollary 3.19]{flag}).

\subsection{Infinite blow-ups and $\Hom^+(\scr A^0[T_{\mathrm{Graph}}], \mathbb{R})$} \label{blowup}

In order to convert the asymptotic result into the exact one, we need to explore a little bit more the connection between blow-ups of a graph and the corresponding algebra homomorphism from $\Hom^+(\scr A^0[T_{\mathrm{Graph}}], \mathbb{R})$ already used in a similar context in \cite[Theorem 4.1]{triangles}. 

For a finite graph $G$ and a positive integer vector $\veca k=( \vecb k(v) \: | \:v\in V(G) )$, we define the {\em blow-up} $G^{(\veca k)}$ of $G$ as the graph with
\begin{eqnarray*}
V(G^{(\veca k)}) &\df& \bigcup_{v\in V(G)} \{v\}\times [\vecb k(v)]\\
E(G^{(\veca k)}) &\df& \set{((v,i),(w,j))}{v\neq w\ \land\ (v,w)\in E(G)}.
\end{eqnarray*}

When all $\vecb k(v)$ are equal to some positive integer $k$, the corresponding blow-up is called {\em balanced} and denoted simply by $G^{(k)}$.

For every graph $H$, it is easy to see that the sequence $\{p(H,G^{(k)})\}_{k \in \mathbb{N}}$ is convergent. It follows~\cite[\S 3]{flag} that there exists a homomorphism $\phi_G \in \Hom^+(\scr A^0[T_{\mathrm{Graph}}], \mathbb{R})$ such that for every graph $H$, we have $\lim_{k \rightarrow \infty} p(H,G^{(k)}) = \phi_G(H)$. Note that since the blow-up of a triangle-free graph is also triangle-free, if $G$ is triangle-free, then actually $\phi_G\in \Hom^+(\scr A^0[T_{\mathrm{TF-Graph}}], \mathbb{R})$.


Let us now give a combinatorial description of $\phi_G$. For a finite graph $H$, let us denote by $s(H,G)$ the number of {\em strong homomorphisms} from $H$ to $G$ that we define as those mappings $\alpha\function{V(H)}{V(G)}$ for which $(\alpha(v),\alpha(w))\in E(G)$ if and only if $(v,w)\in E(H)$.
This notion is a natural hybrid of induced embeddings and graph homomorphisms, and it is a very special case of the notion of trigraph homomorphisms (see e.g. \cite{HeNe}). It is easy to check that
\begin{equation}\label{general_formula}
    \phi_G(H) =\frac{m!}{|\mathrm{Aut}(H)|}\cdot\frac{s(H,G)}{n^m},
\end{equation}
where  $m$ and $n$ respectively denote $|V(H)|$ and $|V(G)|$, and $\mathrm{Aut}(H)$ is the group of automorphisms of $H$.

Let us say that $H$ is {\em twin-free} if no two vertices in $H$ have the same set of neighbors. Every strong homomorphism of a twin-free graph into any other graph is necessarily an induced embedding. Therefore, for twin-free $H$, we have $s(H,G)=p(H,G){n\choose m}|\mathrm{Aut}(H)|$ and \refeq{general_formula} considerably simplifies to
\begin{equation}\label{rigd}
    \phi_G(H) = p(H,G)\cdot \frac{n(n-1)\ldots (n-m+1)}{n^m}.
\end{equation}

For the special case $H=K_r$, this formula was already used in \cite[Section 4.1]{triangles}), and in this paper we are interested in another case,
$$
\phi_{C_5}(C_5)=\frac{5!}{5^5}.
$$

Our approach to extracting exact results from asymptotic ones heavily relies on the fact that $\phi_G$ is a graph invariant. This was first proven by Lov\'asz~\cite{LovaszOperations}. A simple proof of a similar statement in the context of graph limits is given in~\cite[Theorem~5.32]{LovaszBook}.
\begin{theorem} \label{invariance}
Let $G_1$ and $G_2$ be finite graphs with the same number of vertices and such that $\phi_{G_1}=\phi_{G_2}$. Then $G_1$ and $G_2$ are isomorphic.
\end{theorem}

\section{Main results}

Recall \cite[Definitions 5 and 6]{flag} that for a non-degenerate type $\sigma$ in a theory $T$, and $f,g\in\scr A^\sigma[T]$, the inequality $f\leq_\sigma q$ means that $\phi(f)\leq \phi(g)$ for every $\phi\in\Hom^+(\scr A^\sigma[T], \mathbb{R})$: this is the class of all inequalities that hold asymptotically on flags of the given theory \cite[Corollary 3.4]{flag}.
We abbreviate $f\leq_\sigma g$ to $f\leq g$ when $\sigma$ is clear from the context. Our first theorem, which answers the question of Erd\H{o}s, says that in the theory of triangle-free graphs, we have $C_5\leq \frac{5!}{5^5}$. Note that while in the theory of general graphs, the flag $C_5$ corresponds to \emph{induced} pentagons, in the theory of triangle-free graphs, every pentagon is induced.

\begin{theorem} \label{main}
In the theory $T_{\mathrm{TF-Graph}}$, we have
$$
C_5\leq \frac{5!}{5^5}.
$$
\end{theorem}
\begin{proof}
The proof is by a direct computation in the flag algebra $\scr A^0[T_{\mathrm{TF-Graph}}]$ (cf. \cite{turan,HKN} and \cite[Section 4.1]{fdf}). We claim that
\begin{eqnarray}
\nonumber
\lefteqn{62500C_5 + \frac{1097}{12}M_4+\frac{68}{3}C_5^-  + \left(\sum_{i=0}^2 \eval{Q_i^+(\veca g_i^+)}{\sigma_i}\right) + }\\ & & 200\of{\rho-\frac 25}^2 + \eval{Q_1^-(\veca g_1^-)}{\sigma_1} + 158266\eval{(F^{\sigma_2}_{\{1\}}- F^{\sigma_2}_{\{2\}})^2}{\sigma_2}   \leq 2400.
\label{main_inequality}
\end{eqnarray}
The graphs $M_4$ and $C_5^-$ are illustrated on Figure \ref{models}. For the definition of the algebra operations see \cite[Eq. (5)]{flag}, and for the definition of the averaging operator $\eval{\cdot}{}$ see \cite[\S 2.2]{flag}. Let us now define the notations $\veca g_i^{+/-}$ and $Q_i^{+/-}$ in \refeq{main_inequality}.

For a type $\sigma$ of size $k$ and an integer $0\leq j\leq k$, we let
$$
f_j^\sigma\df\sum\set{F^\sigma_V}{V\subseteq [k]\ \text{an independent set of size $j$ in $\sigma$}},
$$
where $F^\sigma_V$ are as defined in Section~\ref{sec:notation}. The vectors $\veca g_i^{+/-}$ are the following tuples of elements from $\scr A^{\sigma_i}_4$:
\begin{eqnarray*}
\veca g_0^+ &\df& (f_1^{\sigma_0}-f_2^{\sigma_0},\ f_1^{\sigma_0}-2f_0^{\sigma_0}+3f_3^{\sigma_0});\\
\veca g_1^+ &\df& (2f^{\sigma_1}_0-f^{\sigma_1}_1,\ f^{\sigma_1}_1-f^{\sigma_1}_2,\ F^{\sigma_1}_{\{3\}});\\
\veca g_2^+ &\df& (6f_0^{\sigma_2}+f_1^{\sigma_2}-4f_2^{\sigma_2},\ 2f^{\sigma_2}_0-2f^{\sigma_2}_2+F^{\sigma_2}_{\{3\}});\\
\veca g_1^- &\df& (F^{\sigma_1}_{\{1\}} - F^{\sigma_1}_{\{2\}},\ F^{\sigma_1}_{\{2,3\}} - F^{\sigma_1}_{\{1,3\}}),
\end{eqnarray*}
and $Q_i^{+/-}$ are positive-definite quadratic forms represented by the following positive-definite matrices:
$$
M_0^+\df \left( \begin {array}{cc} 9760&2252\\\noalign{\medskip}2252&592
\end {array} \right) \hspace{2cm} M_1^+\df \left( \begin {array}{ccc} 13900&-671&-12807\\\noalign{\medskip}-671&
31334&-51136\\\noalign{\medskip}-12807&-51136&98157\end {array}
 \right)
$$
$$
M_2^+\df \left( \begin {array}{cc} 22708&-40788\\\noalign{\medskip}-40788&
78132\end {array} \right) \hspace{2cm} M_1^-\df \left( \begin {array}{cc} 1416&-16452\\\noalign{\medskip}-16452&
256488\end {array} \right).
$$
The inequality \refeq{main_inequality} can be checked by expanding the left-hand side as a linear combination of elements from $\scr M_5$ (that is, triangle-free graphs on 5 vertices -- there are 14 of them) and checking that all coefficients are less or equal than 2400. We verified the inequality using a Maple sheet and a C program which were independently prepared. The C program is available as an ancillary file on the arXiv (arXiv:1102.1634).

It follows from \cite[Theorem 3.14]{flag} that all the summands on the left-hand side of \refeq{main_inequality} are nonnegative as elements of $\mathcal{A}^0[T_{\mathrm{TF-Graph}}]$ which in turn implies that $C_5 \le \frac{2400}{62500}=5!/5^5$.
\end{proof}

\begin{remark}
An explanation of our usage of the $+/-$ superscripts  in the proof of Theorem~\ref{main} can be found in \cite[Section 4]{turan}. Here, the particular choice of subspaces spanned by the vectors $\veca g_i^{+/-}$ is dictated by the same principles as in \cite[Section 4]{turan}.
\end{remark}

\medskip
There are two possible and rather straightforward generalizations of the above problem, neither of which we could find in the literature. One can ask for the maximum number of copies of odd cycles $C_{2\ell+1}$, or for the maximum number of induced copies of odd cycles $C_{2\ell+1}$ in a triangle-free graph of a fixed order. In our pentagon case $\ell=2$ the two questions are the same. Emil Vaughan communicated to us that using the flag algebras method he can prove that the blow-ups of the pentagon and of the heptagon, respectively are asymptotically extremal for the two problems for $\ell=3$.
\medskip

Let us now turn to the question of uniqueness of the original problem.
\begin{theorem} \label{asymptotic_uniqueness}
The homomorphism $\phi_{C_5}$ is the unique element in $\Hom^+(\scr A^0[T_{\mathrm{TF-Graph}}], \mathbb{R})$ that fulfills
\begin{equation}\label{extremal}
    \phi(C_5)=\frac{5!}{5^5}.
\end{equation}
\end{theorem}
\begin{proof}
Fix $\phi\in\Hom^+(\scr A^0[T_{\mathrm{TF-Graph}}], \mathbb{R})$ such that \refeq{extremal} holds. Recall that $P$ is the type of size $5$ based on $C_5$ (see Figure \ref{types}).
Instead of proving directly that $\phi=\phi_{C_5}$, we will argue about the extension ${\bf P}^P$ of $\phi$ and then utilize  ~\eqref{eq:reconstruct}.

Observe that~\refeq{main_inequality} implies that
\begin{equation}\label{phi_zeros}
    \phi(M_4)=\phi(C_5^-)=0.
\end{equation}
Trivially, $\eval{F^P_\emptyset}{P}\leq C_5^-$. As $M_4$ is an induced subgraph of each of the graphs $F^P_{\{1\}}, \ldots, F^P_{\{5\}}$ (viewed as unlabelled graphs), there exists a constant $\alpha>0$ such that $\eval{F^P_{\{i\}}}{P}\leq \alpha M_4$ for every $i\in \mathbb{Z}_5$.  Hence~\refeq{phi_zeros} implies that
\begin{equation}\label{eq:zerozF^P}
\phi(\eval{F^P_\emptyset}{P})=\phi(\eval{F^P_{\{i\}}}{P})=0\;.
\end{equation}

Let $Y$ be the set of those elements $\psi\in \Hom^+(\scr A^P[T_{\mathrm{TF-Graph}}],\mathbb{R})$ for which $\psi(F^P_\emptyset)\neq 0$, or $\psi(F^P_{\{i\}})\neq 0$ for some
$i\in \mathbb{Z}_5$;  note that $Y$ is open. We claim that
\begin{equation}\label{eq:Xzero}
Y\cap S^P(\phi)=\emptyset\;.
\end{equation}
Indeed, let us consider the $P$-flag $f=F^P_\emptyset+\sum_i F^P_{\{i\}}$. By~\eqref{eq:zerozF^P} we have $\phi(\eval{f}P)=0$. Plugging this in~\eqref{eq:defExtMeas}, we have
$$\int_{S^P(\phi)}\psi(f)\; {\bf P}^P(d\psi)=0\;.$$
Further, as $f\ge_P 0$, the integrand $\psi(f)$ is non-negative. Therefore, $\psi(f)=0$ for ${\bf P}^P$-almost all elements $\psi$ and, since $Y$ is open, ~\eqref{eq:Xzero} follows.

Pick an arbitrary $\phi^P\in S^P(\phi)$. Let us examine $\phi^P(F^P_V)$ for flags $F_V^P\in\scr F^P_6$.
Since $V$ is an independent set in $C_5$, we have $|V|\leq 2$, and moreover, if $|V|=2$, then $V=\{i-1,i+1\}$ for some $i\in \mathbb{Z}_5$.
As $\phi^P\not\in Y$, we have $\phi^P(F^P_\emptyset)=\phi^P(F^P_{\{i\}})=0$.

In other words, $\phi^P(F^P_V)$ can be non-zero only when $V=\{i-1,i+1\}$ for some $i\in \mathbb{Z}_5$. Define $H^P_i\df F^P_{\{i-1,i+1\}}$. We have $\sum_{i\in \mathbb{Z}_5}\phi(H^P_i)=1$, and thus the inequality of aritmetic and geometric means gives that
\begin{equation}\label{eq:AG}
\prod_{i\in \mathbb{Z}_5}\phi(H^P_i)\le 5^{-5}\;,
\end{equation}
with equality only when $\phi(H^P_1)=\ldots=\phi(H^P_5)=1/5$.

Recall that $\pi^P(C_5)$ can be represented as the sum of those $F=(G,\theta)\in \mathcal{F}^P_{10}$ for which the unlabeled vertices form a copy of $C_5$, say $V(G)\setminus\im(\theta)=\{v_1,v_2,\ldots, v_5\}$ where $v_j$ is adjacent to $v_{j-1}$ and $v_{j+1}$. By the above discussion, non-zero contributions to $\phi^P(\pi^P(C_5))$ can be made only by those $F$ for which every $v_j$ is adjacent to $\theta(i(j)-1)$ and $\theta(i(j)+1)$ for some choice of $i(j)\in \mathbb{Z}_5$. Since $G$ is triangle-free, the mapping $j\mapsto i(j)$  defines a graph homomorphism of the pentagon into itself, and since there are no such graph homomorphisms other than isomorphisms, we may assume without loss of generality that every $v_j$ is adjacent to $\theta(j-1)$ and $\theta(j+1)$.
In other words, $\phi^P(\pi^P(C_5))=\phi^P((C_5^{(2)})^P)$, where $(C_5^{(2)})^P$ is the uniquely defined $P$-flag based on  $C_5^{(2)}$, the blow-up of the pentagon.

Observe that
\begin{equation}\label{no_slackness}
  (C_5^{(2)})^P \:\leq_P\: 5!\cdot \prod_{i\in \mathbb{Z}_5}H_i^P\;,
\end{equation}
and thus by ~\eqref{eq:reconstruct} and \eqref{eq:AG} we have the following chain of inequalities
$$
  \frac{5!}{5^5}=\phi(C_5)= \phi^P(\pi^P(C_5))= \phi^P((C_5^{(2)})^P)\le 5! \prod_{i\in \mathbb{Z}_5} \phi^P(H^P_i)\le \frac{5!}{5^5}\;.
$$
Consequently, ~\eqref{eq:AG} must actually be an equality, and therefore $\phi^P(H^P_1)=\ldots=\phi^P(H^P_5)=1/5$. Further, there is no slackness in~\refeq{no_slackness}, i.e.,
\begin{equation}\label{no_slackness2}
    \phi^P\left(5!\cdot \prod_{i\in \mathbb{Z}_5}H_i^P-(C_5^{(2)})^P \right)=0.
\end{equation}
This equality allows us to completely describe the behavior of $\phi^P$ also on $\scr F^P_7$. For $i,j \in \mathbb{Z}_5$, let $H^P_{ij} \in \scr F_7^P$ be defined by adding two unlabeled non-adjacent vertices to $P$ and connecting one of them to $\theta(i-1)$ and $\theta(i+1)$ and the other to $\theta(j-1)$ and $\theta(j+1)$. Note that if $i=j$ or $(i,j)\not \in E(P)$, then the product $H^P_iH^P_j$ is equal to $q_{ij}H^P_{ij}$, where $q_{ij}=1$ if $i=j$, and $q_{ij}=1/2$
otherwise (since adding an edge between the two unlabeled vertices would have created a triangle). Hence, in this case $\phi^P(H^P_{ij})=\frac{1}{25q_{ij}}$. On the other hand,
if $(i,j)\in E(P)$, then $$5!\cdot H_{ij}^P\cdot\prod_{k\in \mathbb{Z}_5\setminus \{i,j\}}H^P_k \:\leq_P\: 5!\cdot \prod_{i\in \mathbb{Z}_5}H_i^P-(C_5^{(2)})^P,$$ 
which together with \refeq{no_slackness2} implies that $\phi^P(H_{ij}^P)=0$. It
follows that $\phi^P(G^P_{ij})=\frac{1}{25q_{ij}}$, where $G^P_{ij}$ is defined similar to $H^P_{ij}$ with the difference that now there is an edge between the unlabeled vertices.
As $\sum \phi^P(H^P_{ij}) + \sum \phi^P(G^P_{ij}) =1$, we have
\begin{equation}\label{eq:us0}
\phi^P(F)=0  \mbox{ for each $F\in \mathcal{F}^P_7\setminus\bigcup_{i,j\in
\mathbb{Z}_5} \{G^P_{ij},H^P_{ij}\}$.}
\end{equation}

Let $H$ be now an arbitrary triangle-free graph on $n$ vertices. Similarly to the above we can expand $\pi^P(H)$ in $\mathcal{A}^P_{5+n}$.
For a homomorphism $\alpha:H\rightarrow \mathbb{Z}_5$ of $H$ to $C_5$ (with its vertices labeled cyclically), we write $F^P_\alpha$ for the
$P$-flag $(G,\theta)\in\mathcal{F}^P_{5+n}$ where the unlabeled vertices $V'\df V(G)\setminus \im (\theta)$ induce a copy of $H$, and  each vertex $v\in V'$ is adjacent only to
$\theta(\alpha(v)-1)$ and $\theta(\alpha(v)+1)$. We claim that $\phi^P$ evaluates to zero at any term $F\in \scr F^P_{5+n}$ in the expansion of $\pi^P(H)$, unless $F=F^P_\alpha$
for some homomorphism $\alpha:H\rightarrow P$. Indeed, the particular case when $H$ is an edge is shown in~\eqref{eq:us0}, and the general case follows by the same reasoning.
Furthermore, as $\phi^P(H^P_{ij})=0$ for $(i,j)\in E(P)$, we actually have that $\alpha$ must be a strong homomorphism in this case. Observe that
\begin{equation}\label{phiFh} \frac{\phi^P(F^P_\alpha)}{c_\alpha}=\phi^P\left(\prod_{v\in V(H)}H^P_{h(v)}\right)=5^{-n} \end{equation} for each strong homomorphism
$\alpha:H\rightarrow P$, where $c_\alpha$ is the multinomial coefficient, $$c_\alpha\df{n\choose |h^{-1}(1)|,\;|h^{-1}(2)|,\;|h^{-1}(3)|,\;|h^{-1}(4)|,\;|h^{-1}(5)|}.$$

Now, $\phi(H)=\phi^P(\pi^P(H))=\sum_\alpha\phi^P(F^P_\alpha)=5^{-n}\cdot \sum_\alpha c_\alpha$ (the summation is taken over all strong homomorphisms from $H$ to
$C_5$) that can be easily seen to coincide with the value $\phi_{C_5}(H)$ as given by \eqref{general_formula}.
This finishes the proof.
\end{proof}

The upper bound on the number of pentagons can be derived from Theorem~\ref{main}
on the infinite blow-up $\phi_G\in \Hom^+(\scr A^0[T_{\mathrm{TF-Graph}}], \mathbb{R})$.
Furthermore, Theorems~\ref{invariance} and~\ref{asymptotic_uniqueness} show that
the equality in the statement can be obtained only when $n$ is divisible by five and
$G$ is the balanced blow-up of the pentagon.

\begin{corollary}
\label{cor-main}
Every $n$-vertex triangle-free graph $G$ contains at most $(n/5)^5$ pentagons.
Moreover, the equality is attained only when $n$ is divisible by five and
$G$ is the balanced blow-up of the pentagon.
\end{corollary}

\bigskip

The bound attained by Corollary \ref{cor-main} is not tight when $n$ is not divisible by 5. More specifically, let $n=5\ell+a\ (0\leq a\leq 4)$, then the number of pentagons in an
almost balanced blow-up of $C_5$ with $n$ vertices\footnote{when $a=2,3$ there are two non-isomorphic almost balanced blow-ups of $C_5$ with $n$ vertices} is equal to $\chi(n)\df \ell^{5-a}(\ell+1)^a$.

\begin{conjecture} \label{conj:one}
Every triangle-free graph on $n$ vertices contains at most $\chi(n)$ pentagons.
\end{conjecture}

The original version of this paper claimed to resolve Conjecture~\ref{conj:one}, but the proof contained a mistake that we were not able to fix. As we noted in introduction, Michael \cite{Mic} observed that for $n=8$ there exists a sporadic example with $\chi(8)=8$ pentagons. In fact, \cite{Mic} also conjectured that this
is the {\em only} sporadic example, which, in particular, would imply Conjecture~\ref{conj:one}.

We use a stability argument to settle Conjecture~\ref{conj:one} for sufficiently large $n$ in Theorem~\ref{thm:exactGeneral} below.
\section{Exact bound}\label{sec:Exact}

We define the \emph{cut norm} of an $n \times n$ matrix $A$  by
$$\|A\|_\square := \frac{1}{n^2} \max_{S,T \subseteq [n]} \left|\sum_{i \in S, j \in T} A_{ij} \right|.$$
For two graphs $G_1$ and $G_2$ on the same set of vertices $[n]$, we define their cut distance as
$$d_\square(G_1,G_2) = \|A_{G_1}- A_{G_2}\|_\square,$$
where $A_{G_1}$ and $A_{G_2}$ denote respectively the adjacency matrices of $G_1$ and $G_2$.
If $G_1$ and $G_2$ are unlabeled  graphs on different vertex sets of the same cardinality
$n$, then we define their distance by
$$\widehat{\delta}_\square(G_1,G_2) = \min_{\tilde{G}_1,\tilde{G_2}} d_\square(\tilde{G}_1,\tilde{G}_2),$$
where $\tilde{G}_1$ and $\tilde{G}_2$ range over all labellings of $G_1$ and $G_2$ by $[n]$ respectively.
Finally, let $G_1$ and $G_2$ be graphs with $n_1$ and $n_2$ vertices, respectively. Note that for every positive integer $k$,
the blow-up graphs  $G_1^{(n_2 k)}$ and $G_2^{(n_1 k)}$ have the same number of vertices. So we can define
$$\delta_\square(G_1,G_2) = \lim_{k \to \infty} \widehat{\delta}_\square(G_1^{(n_2 k)},G_2^{(n_1 k)}).$$
The function $\delta_\square$ is only a pseudometric, not a true metric, because $\delta_\square(G,G')$
may be zero for different graphs $G$ and $G'$. In fact, it is an easy consequence of Theorem 
\ref{invariance} and Theorem \ref{thm:convergenceCut} below that $\delta_\square(G,G')=0$ if and only if $G^{(k)} \cong G'^{(k')}$
for some  integers $k,k'>0$. The following theorem shows the relevance of this distance to the context of this paper.

\begin{theorem}[{\cite[Theorem 2.6]{BCLSV:ConvergentI}}§]
\label{thm:convergenceCut}
A sequence of graphs $(G_n)_{n=1}^\infty$ converges in the $\delta_\square$ distance if and only if the sequence $(\phi_{G_n})_{i=1}^n$ converges.
\end{theorem}

We will also need the following fact proven by Alon (see~\cite[Theorem~9.24]{LovaszBook}). If $G_1$ and $G_2$ are two graphs on $n$ vertices, then
\begin{equation}
\label{eq:Alon}
\widehat{\delta}_\square(G_1,G_2) \le \delta_\square(G_1,G_2) + \frac{17}{\sqrt{\log n}}.
\end{equation}

\begin{theorem}
\label{thm:exactGeneral}
   There exists an integer $n_0$ such that any triangle-free graph with $n\geq n_0$ vertices and at least $\chi(n)$ pentagons must be an almost balanced blow-up of $C_5$.
\end{theorem}
\begin{proof}
Let $\delta_1>0$ be a fixed constant  chosen to be sufficiently small to satisfy the inequalities throughout the proof.  Let $G$ be a triangle-free graph on $n$ vertices containing the maximum number of pentagons. Throughout the proof, we will always assume
that $n$ is sufficiently large. We will use some additional notation in the proof below. For $U,V \subseteq V(G)$, we will denote by $G[U]$ the subgraph of $G$ induced by $U$ and we will denote by $G[U,V]$ the  induced bipartite subgraph of $G$  with parts $U$ and $V$. For $v \in V(G)$ we denote by $\deg_U(v)$ the number of neighbors of $v$ in $U$.

By Theorem~\ref{thm:convergenceCut}~and Theorem~\ref{asymptotic_uniqueness}, we have  $\delta_\square(G,C_5) \le \delta_1/4$. Let $\mathbf{k}$ be so that $C_5^{(\mathbf{k})}$ is an almost balanced blow-up of $C_5$ on $n$ vertices. Note that
$$\delta_\square(G,C_5^{(\mathbf{k})}) \le \delta_\square(G,C_5^{(5\lfloor n/5 \rfloor)}) +\delta_\square(C_5^{(\mathbf{k})},C_5^{(5\lfloor n/5 \rfloor)}) = \delta_\square(G,C_5) +O(1/n).$$
Combining this with (\ref{eq:Alon}), we conclude that if $n$ is sufficiently large, then $\widehat{\delta}_\square(G,C_5^{(\mathbf{k})}) \le \delta_1/2.$
In particular there exists a partition $A_1,A_2,\ldots,A_5$ of $V(G)$ such that for every $i \in [5]$, it holds that $||A_i| - n/5| \leq 1$ and
\begin{equation}
\label{eq:exact1}
|E(G[A_i,A_{i+1}])| \ge |A_{i}||A_{i+1}|-\delta_1 n^2/2.
\end{equation}

Set $\delta_2=3\sqrt{\delta_1}$, and for every $i \in [5]$, let $B_i$ be the set of vertices $v \in A_i$ such that
$\deg_{A_{i-1} \cup A_{i+1}}(v) \leq |A_{i-1} \cup A_{i+1}| - \delta_2 n$. By \eqref{eq:exact1} we have
$$|E(G[A_i,A_{i-1} \cup A_{i+1}])| \ge |A_{i}||A_{i-1} \cup A_{i+1}|-\delta_1 n^2,$$
which implies $|B_i| \le \frac{\delta_1}{\delta_2} n \le \frac{\delta_2}{5} n$. Let $A_i' = A_i \setminus B_i$. We claim that
\begin{itemize}
\item[(a)] $||A_i' |- n/5| \leq \frac{\delta_2}{5}n+1\leq \delta_2n$.
\item[(b)] $\deg_{A'_{i-1} \cup A'_{i+1}}(v) \geq  |A_{i-1} \cup A_{i+1}| - \frac{7}{5}\delta_2 n \ge \frac{2n}{5} - 2 \delta_2 n$ for every $v \in A'_i$.
\item[(c)] $\deg_{A'_{i \pm 1}}(v) \ge |A_{i \pm 1}(v)| -  \frac{7}{5} \delta_2 n \ge \frac{n}{5} -2\delta_2n$ for every $v \in A'_i$.
\item[(d)] $E(G[A_i'])=E(G[A_i',A_{i+2}'])=\emptyset$.
\end{itemize}

Conditions (a) and (b) are immediate consequences of the bound on the size of $B_i$. By (b), for $v \in A'_i$, we have $\deg_{A'_{i \pm 1}}(v) \ge |A_{i \pm 1}| - \frac{7}{5} \delta_2 n $, which verifies (c). To  prove (d) note that for $\delta_2$ sufficiently small, every pair of vertices in $A_i'$ have (many) common neighbors. Hence, as $G$ is triangle-free, $E(G[A_i'])=\emptyset$. It follows similarly that $E(G[A_i',A_{i+2}'])= \emptyset$.

In the next step we eliminate the vertices in $B\df \bigcup_{i=1}^5 B_i$. Consider an arbitrary $v \in B$. We would either add $v$ to one of the parts of $A'_1,\ldots,A'_5$ maintaining the conditions on $A_{i}'$ established above (possibly with a worse constant) or show that $v$ is in few pentagons and  can be replaced by another vertex, while the number of pentagons is increased.

Let $p(v)$ denote the number of pentagons in $G$ containing $v.$ Consider first deleting the vertex $v$ and adding a new vertex to $G$ joined to every vertex in $A_1'$ and $A_3'$. By (d) the graph remains triangle-free. By (a) and (c), the new vertex is in at least
$$\left(\frac{n}{5} -\delta_2n \right)\left(\frac{n}{5} -2\delta_2n \right)^3 \geq \frac{n^4}{5^4} - 7\delta_2 n^4 =: p_1$$
pentagons for $\delta_2$ sufficiently small. As the number of pentagons in $G$ is maximum among all graphs on $n$ vertices, we conclude that $p(v) \geq p_1$.

Consider a pentagon containing $v$ and four vertices in $V(G)-B$, so that at least two of them lie in the same $A'_i$. It is not hard to verify that such a pentagon must contain a pair of non-adjacent vertices one in $A'_i$ and another in $A'_{i+1}$ for some $i \in [5]$. Thus there are at most $4 \cdot \frac{7}{5}\delta_2 n^4 \leq 6 \delta_2n^4$ such pentagons by (b). Also since $|B| \le \delta_2n$, there are at most $8\delta_2n^4$ pentagons containing $v$ and another vertex in $B$.
Let $x_i = \deg_{A'_i}(v)$. Note that if $x_i > 2\delta_2n$ for some $i$, then $x_{i+1}=x_{i-1}=0$, as  a vertex in $A'_{i \pm 1}$ can have at most $2\delta_2n$ non-neighbors in $A'_i$ by (c) and $G$ is triangle-free.
Therefore, if for every $i \in [5]$, we have $x_i \leq 2\delta_2n$ or $x_{i+2} \leq 2\delta_2n$, then we get a contradiction:
\begin{align*}
p(v) \leq 14\delta_2n^4 + 10\delta_2n\left( \frac{n}{5} + \delta_2n\right)^3 < p_1.
\end{align*}
We conclude $x_i,x_{i+2} > 2\delta_2n$ for some $i \in [5]$, and we have $x_{i-1}=x_{i+1}=x_{i+3}=0$ by the observation above. Let $\delta_3=5^3 \cdot 22\delta_2$. Suppose that
$x_i \leq \frac{n}{5}-\delta_3n$ or $x_{i+2} \leq \frac{n}{5}-\delta_3n$. Repeating the calculation above we have
\begin{align*}
p(v) &\leq 14\delta_2n^4 + \left(  \frac{n}{5} - \delta_3n \right)\left( \frac{n}{5} + \delta_2n\right)^3 <  14\delta_2n^4 + \left(  \frac{n}{5} - \delta_3n \right)\left( \frac{n^3}{5^3} + \delta_2n^3\right)\\ &< \frac{n^4}{5^4}  - \left(\frac{\delta_3}{5^3} -15 \delta_2 \right)n^4   = p_1,
\end{align*}
where the intermediate inequalities hold for $\delta_2$ sufficiently small and the last identity is by the choice of $\delta_3$. Thus $x_i,x_{i+2} \geq \frac{n}{5}-\delta_3n$ for some $i \in [5]$. We add $v$ to $A'_{i+1}$. We repeat the same procedure for every vertex of $B$.

As a result of the procedure described in the preceding paragraph, we obtain a partition $A''_1,A''_2, \ldots,A''_5$ of $V(G)$ such that  $||A''_i|-n/5|\leq 2\delta_2 n \leq \delta_3n$, $\deg_{A''_{i-1}}(v) \geq n/5-\delta_3n$ and $\deg_{A''_{i+1}}(v) \geq n/5-\delta_3n$ for every $i \in [5]$ and every $v \in A''_i$. As in (d) it follows that  $E(G[A''_i])=E(G[A''_i,A_{i+2}''])= \emptyset$ for every $i \in [5]$, if  $\delta_3$ is sufficiently small. Thus every pentagon in $G$ must contain a vertex from each of the sets $A''_1,A''_2, \ldots,A''_5$. Consequently, the number of pentagons in $G$ is upper bounded by $\prod_{i=1}^5|A''_i|$ and the equality holds if and only if every vertex of $A''_i$ is joined to every vertex in $A''_{i+1}$ for every $i \in [5]$. It is easy to see that the above product is maximized when $||A''_i| -|A''_j|| \leq 1$ for every $i,j \in [5]$. Thus $G$ must be an almost balanced blow-up of $C_5$, as desired.
 \end{proof}

\section*{Acknowledgement}
When preparing the final version of this paper we learned that
Andrzej Grzesik~\cite{Grzesik} independently proved Theorem~\ref{main}.

We thank Oleg Pikhurko for pointing out the work of Zolt\'an
F\"uredi to us. We also thank T.~S.~Michael and Humberto Naves
for pointing out the flaw in the original version of the manuscript, and referees for their comments.

\bibliographystyle{alpha}
\bibliography{erdos}

\end{document}